\documentclass[12pt,a4paper]{amsart}
\usepackage{amsfonts}
\usepackage{amssymb}
\usepackage{amsmath}
\usepackage{amsthm}
\usepackage{cite}
\usepackage{enumitem}
\usepackage{tikz}

\usepackage{caption}
\usepackage{subfigure}

\usepackage{MnSymbol}

\theoremstyle{plain}
\newtheorem{thm}{Theorem} 

\newtheorem{lem}{Lemma}
\newtheorem{prop}{Proposition}

\providecommand{\skp}[2]{\langle#1,#2\rangle}

\providecommand{\sm}{\setminus}
\providecommand{\N}{\mathbb{N}}
\providecommand{\R}{\mathbb{R}}
\providecommand{\Z}{\mathbb{Z}}
\providecommand{\C}{\mathbb{C}}
\providecommand{\eps}{\varepsilon}
\providecommand{\ov}{\overline}

\providecommand{\skp}[2]{\langle#1,#2\rangle}

\DeclareMathOperator{\sign}{sign}

\DeclareMathOperator{\ind}{ind}
\DeclareMathOperator{\ran}{ran}
\DeclareMathOperator{\Id}{Id}
\DeclareMathOperator{\Real}{Re}
\DeclareMathOperator{\Imag}{Im}
\DeclareMathOperator{\dist}{dist}
\DeclareMathOperator{\spa}{span}
\DeclareMathOperator{\pr}{pr}
\renewcommand{\qed}{\hfill $\Box$}

\setitemize{itemsep=+2pt}
\setenumerate{itemsep=+2pt}

\begin{document}

\allowdisplaybreaks

\title{Global secondary bifurcation, symmetry breaking and period-doubling}

\author{Rainer Mandel}
\address{R. Mandel \hfill\break
Karlsruhe Institute of Technology \hfill\break
Institute for Analysis \hfill\break
Englerstra{\ss}e 2 \hfill\break
D-76131 Karlsruhe, Germany}
\email{Rainer.Mandel@kit.edu}
\date{}


\begin{abstract}
  In this paper we provide a criterion for global secondary bifurcation via symmetry breaking. As an
  application, the occurrence of period-doubling bifurcations for the Lugiato-Lefever equation is proved.
\end{abstract}

\maketitle
\allowdisplaybreaks

\section{Introduction}

  The aim of this paper is to provide a sufficient condition for global secondary bifurcation via symmetry
  breaking for equations of the form 
  \begin{equation} \label{eq}
    F(x,\lambda)=0
  \end{equation}
  where $x\in X$ belongs to a Banach space and $\lambda\in\R$ is a real parameter. Bifurcation theory is about
  finding solutions near a given family of trivial solutions of \eqref{eq}. For instance, if
  $F(0,\lambda)=0$ for all $\lambda\in\R$ then the trivial solution family is given by $\{(0,\lambda):\lambda\in
  \R\} \subset X\times\R$. More generally, if $\mathcal T\subset X\times\R$ is a family of 
  solutions then $(x,\lambda)\in\mathcal T$ is a bifurcation point with respect to $\mathcal T$ if
  there is a sequence of solutions $(x_n,\lambda_n)\notin \mathcal T$ converging to $(x,\lambda)$. In this
  case one speaks of (primary) bifurcation with respect to $\mathcal T$ and there are many
  powerful theorems that allow to detect such bifurcations under suitable assumptions on $F$. Examples for
  such theorems are the celebrated bifurcation results due to Marino, B\"ohme \cite{Marino,Boehme}, Crandall,
  Rabinowitz \cite{CR_bif_simple} or Krasnoselski, Rabinowitz
  \cite{Kras:Topological,Rab:global}. The latter ones even allow to conclude that the bifurcating solutions
  lie on a nontrivial connected set of solutions $\mathcal C\subset X\times\R$. Such a set is sometimes
  called a primary solution branch.
  
  \medskip 
  
  Our interest lies in secondary bifurcation, which we define, roughly speaking, as bifurcation with
  respect to such primary solution branches. We refer to
  Section~\ref{sec:symbreak} for precise definitions. Our main result (Theorem~\ref{thm_symbreak}) will
  provide sufficient conditions for the occurrence of secondary bifurcation without any explicit
  knowledge of the primary branch. As a byproduct, this secondary bifurcation comes with the phenomenon of
  symmetry-breaking and it will be shown to be global in a  sense that we will make
  precise later. As far as we know, such an analysis has not been done before. Actually, very few analytical papers
  deal with secondary bifurcations. In the paper \cite{BKR_sec_bif} by Bauer, Keller and Reiss it is outlined
  how local secondary bifurcations may occur for eqfations with two real parameters near a degenerate
  trivial solution. However, their approach is local in nature and it is not rigorously stated nor
  proved in an abstract setting, which makes their results hardly comparable to those
  that we present in this paper. One example for a secondary bifurcation analysis based on an almost
  explicit knowledge of the primary solution branch is  
  presented in the paper~\cite{KuMoTsYo_2ndarybif} in the context of a one-dimensional nonlocal Allen-Cahn
  equation. An interesting result related to the nonexistence of secondary bifurcation points is contained
  in~\cite{Miy_nonexistence}.
    
  \medskip
  
  The literature on symmetry breaking results is much larger and we mention at least some of the
  available results. We focus on those that apply
  to the study of nonradial solutions of nonlinear elliptic PDEs of the form
  \begin{equation} \label{eq:LaneEmden}
    -\Delta u  = f(u,\lambda) \quad\text{in }\Omega,\qquad u\in H_0^1(\Omega)
  \end{equation}
  where $\Omega$ is an annulus in $\R^n$. In the case of a ball the celebrated symmetry result of
  Gidas, Ni and Nirenberg \cite{GNN_symmetry} shows that all positive solutions of \eqref{eq:LaneEmden} are
  automatically radially symmetric if $f(\cdot,\lambda)$ is continuously differentiable. The corresponding statement
  for annuli is not true for all $f$, as was shown variationally by Coffman~\cite{Coffman:nonlinear_bvp} for
  $f(z,\lambda)= -z+z^{2m+1}$ and $m\in\N$. Srikanth \cite{Sri:symbreak} considered symmetry breaking for
  \eqref{eq:LaneEmden} when the nonlinearity is given by $f(z,\lambda)=|z|^{p-1}z+\lambda z$ with
  $p>1,\lambda\in\R$ and annuli $\Omega$ such that the inner radius almost equals the
  outer one.
  Computing the Leray-Schauder index along the uniquely determined curve of positive radial solutions he discovered
  nonradial solutions via symmetry breaking bifurcation from this curve. Similarly, much is known
  about the local and global shape of the nonradial solutions bifurcating from the curve of radial solutions
  for the Gelfand problem~\eqref{eq:LaneEmden} with $f(z,\lambda)=\lambda e^z$, 
  see~\cite{Lin:nonradial_bifurcation} (Theorem~4.4) and \cite{Dan_global_breaking} (Theorem~2). Notice that
  nonradial bifurcation results from radial solutions of \eqref{eq:LaneEmden} are also available on balls (see
  for instance Theorem~2.1 in~\cite{Cer:sym_bre}  or Theorem~5.4 in~\cite{SmoWas_symbre}), but the
  bifurcation points have to be sign-changing radial solutions by the above-mentioned symmetry result of
  Gidas, Ni and Nirenberg.
  A symmetry breaking result for equations of the form \eqref{eq:LaneEmden} with a forcing term is due to
  Dancer, see Theorem~2 in~\cite{Dan:breaking_of_symmetries}. Let us finally mention an interesting recent
  contribution showing a completely different way of symmetry breaking in the context of nonlinear elliptic
  systems via variational methods~\cite{BrClMar:symbrea}.

  \medskip


   Let us briefly describe how this paper is organized. In the following section we recall Rabinowitz' global
   bifurcation theorem along with a refinement due to Dancer in a slightly more general framework than usual.
   Based on this theorem we will state and prove our main result on symmetry breaking via secondary bifurcation in Section~\ref{sec:symbreak}.
   In Section~\ref{sec:Applications}, we apply these abstract results in order to detect
   period-doubling secondary bifurcations for the Lugiato-Lefever equation. Actually, this application
   motivates the above-mentioned generalization of Rabinowitz' theorem. The proof of this
   result closely follows the original one and is therefore postponed to Appendix A. In Appendix B we comment
   on the regularity assumptions on $F$ that are used in the proof.  
   We emphasize that our secondary bifurcation analysis will not rely on local
   considerations or on the fact that the primary bifurcation branch is actually explicitly known. 
   In particular, our results on period-doubling bifurcation for the Lugiato-Lefever equation can not be
   proved by means of a local period-doubling bifurcation result such as Theorem~I.14.2. in Kielh\"ofer's
   book~\cite{Kielhoefer}.
  
  \section{On Rabinowitz' Global Bifurcation Theorem} \label{sec:rabinowitz}
  
  In Theorem~1.3 of the paper~\cite{Rab:global} Rabinowitz studied the equation $F(x,\lambda)=0$ where
  $F(x,\lambda)=x-\lambda Lx-H(x,\lambda)$, $L$ is a compact linear map and $H:X\times\R\to X$ is compact and
  continuous with $H(x,\lambda)=o(\|x\|)$ locally uniformly with respect to $\lambda$ as $x\to 0$. Roughly
  speaking, he globalized Krasnoselski's Bifurcation Theorem~\cite{Kras:Topological} by proving that solutions bifurcating
  from the trivial solution $x=0$ at some characterictic value $\lambda_0$ of $L$ of odd algebraic
  multiplicity lie on a continuum of solutions $\mathcal C\subset X\times\R$ that is unbounded or returns to
  the trivial solution family at some other characteristic value of $L$. Recall that the characteristic values
  of $L$ are the reciprocals of its eigenvalues.
  Later, Dancer remarked that if $\mathcal C$ is bounded and intersects the trivial solution family at
  mutually different $\lambda_0,\ldots,\lambda_k$, then the jumps of the Leray--Schauder indices at the
  trivial solutions $(0,\lambda_0),\ldots,(0,\lambda_k)$ have to sum up to zero, see Theorem~1
  in~\cite{Dancer_On_the_structure}. In particular, $\mathcal C$ contains an even number of
  trivial solutions $(0,\lambda_j)$ where $\lambda_j$ is a characteristic value of odd multiplicity.   
  Both Rabinowitz' and Dancer's contributions are fundamental for the rest of  this paper.
    
  \medskip
  
  In our result on secondary bifurcations we want to make use of the above-mentioned results for equations
  $F(x,\lambda)=0$ in a more general setting, where $F$ and the trivial   solution
  family $\mathcal T\subset X\times\R$ satisfy less restrictive assumptions. This is motivated by our
  application to the Lugiato--Lefever equation that we will discuss in Section~\ref{sec:Applications}. 
  We will prove these results under the following assumptions on $F$ and $\mathcal T$:
  \begin{itemize}
    \item[(A1)] $F\in C(X\times\R,X)$ is a compact perturbation of the identity,
    \item[(A2)] $\mathcal T\subset X\times\R$ is a closed embedded 1-submanifold of class $C^1$ such that
    $F|_{\mathcal T}=0$, $F$ is locally uniformly differentiable along $\mathcal T$ with $F'\in C(\mathcal
    T,X\times\R)$ and the subset of degenerate solutions on $\mathcal T$ is discrete.
  \end{itemize}
  Several remarks are in order. Firstly, (A1) means that the map $(x,\lambda)\mapsto x-F(x,\lambda)$ is
  continuous and compact on $X\times\R$. This ensures that Leray-Schauder degree theory is applicable so that
  the main degree-theoretic ideas of Rabinowitz' proof carry over. In the case $\mathcal
  T=\{(0,\lambda):\lambda\in \R\}$ this is well-known, see for instance Theorem II.3.3
  in~\cite{Kielhoefer}. Concerning (A2), we first point out that $\mathcal
  T$ need not be unbounded; it may as well be a simple closed $C^1$-curve in $X\times\R$. We say that a point
  $(x_0,\lambda_0)\in\mathcal T$ is degenerate if $\mathcal T$ is locally
  parametrized by a regular curve $(\bar x,\bar \lambda):(t_0-\eps,t_0+\eps)\to X\times\R$ such that $(\bar
  x(t_0),\bar \lambda(t_0))=(x_0,\lambda_0)$ and 
  $$
    \ker(F'(x_0,\lambda_0))\varsupsetneqq \spa\{(\bar x'(t_0),\bar \lambda'(t_0))\}.
  $$ 
  Here, $F':X\times\R\to X$ stands for the Fr\'{e}chet derivative of $F$. Notice that this
  notion of degeneracy does not depend on the chosen parametrization. Locally uniform differentiability 
  along $\mathcal T$ means that $F'$ exists at all elements of $\mathcal T$ such that for all local
  $C^1-$parametrizations $(\bar x,\bar\lambda):I\to\R$ of $\mathcal T$ and all convergent sequences
  $(x_n),(\lambda_n),(t_n)$ with $(x_n,\lambda_n)-(\bar x(t_n),\bar\lambda (t_n))\to (0,0)$ we have 
  $$ 
    \frac{\|F(x_n,\lambda_n)-F(\bar x(t_n),\bar \lambda(t_n)) - F'(\bar x(t_n),\bar\lambda(t_n))[(x_n-\bar
    x(t_n),\lambda_n-\bar\lambda(t_n))]\|}{
    \|x_n-\bar x(t_n)\|+|\lambda_n-\bar\lambda(t_n)|} \to 0
    \quad\text{as }n\to\infty. 
  $$
  For instance, this condition holds provided $F$ is continuously differentiable in an open neighbourhood of
  $\mathcal T$. This regularity assumption on $F$ allows to conclude that bifurcation points
  with respect to $\mathcal T$ are necessarily degenerate and therefore do not accumulate, which will be
  essential in Theorem~\ref{thm_symbreak}. For the convenience of the reader we include a proof. 
  
  \begin{prop}\label{prop:unif_diff} 
    Assume (A1),(A2). Then the set of bifurcation points with respect to $\mathcal T$ is discrete.
  \end{prop} 
  \begin{proof}
    By (A2), it suffices to show that every bifurcation point with respect to $\mathcal T$ is degenerate.
    In the notation from above let $(\bar x(t^*),\bar\lambda(t^*))\in\mathcal T$ be such a bifurcation point
    and choose the subspace $W$ such that $X\times\R= \spa\{\psi\}\oplus W$, 
    where $\psi:=(\bar x'(t^*),\bar \lambda'(t^*))$. So there are $C^1$-functions $\bar w$ and $\bar \mu$ with range in $W$ and $\R$,
    respectively, such that $(\bar x(t),\bar \lambda(t))= \bar w(t)+\bar \mu(t)\psi$ for $t$ close to $t^*$.
    By construction of $\psi$ we then have $\bar\mu'(t^*)=1$. Since  $(\bar x(t^*),\bar\lambda(t^*))\in\mathcal
    T$ is a bifurcation point, this implies that there is a sequence $(t_n)$ converging to $t^*$ and
    $w_n\in W\sm\{0\}$ such that the nontrivial solutions $\bar w(t_n)+w_n+\bar\mu(t_n)\psi$ converge to 
    $\bar w(t^*)+\bar\mu(t^*)\psi$ as $n\to\infty$. So the function $G:W\times\R\to X, (w,t)\mapsto F(\bar
    w(t)+w+\bar\mu(t)\psi)$ satisfies $G(0,t)=0$ for $t$ close to $t^*$ as well as
    $$
      0 
      = G(w_n,t_n)
      = G(w_n,t_n)-G(0,t_n)
      = G_w(0,t_n)[w_n] + o(\|w_n\|) 
    $$
    by the uniform differentiability of $F$ along $\mathcal T$. From this we get
    $G_w(0,t_n)[w_n/\|w_n\|]\to 0$ and hence $G_w(0,t^*)[w_n/\|w_n\|]\to 0$ by continuity of $t\mapsto
    G_w(0,t)$ at $t^*$.
    Exploiting that $G_w(0,t^*)$ is a compact perturbation of the identity, we find that a
    subsequence of $(w_n/\|w_n\|)$ converges to some nontrivial $\xi\in W$ in the kernel of 
    $F'(\bar x(t^*),\bar\lambda(t^*))$, hence $(\bar x(t^*),\bar\lambda(t^*))$ is degenerate.  
  \end{proof}
  
  Notice that the statement of Proposition~\ref{prop:unif_diff} need not be true if only $F'\in C(\mathcal
  T,X)$ is assumed as in Kielh\"ofer's
  version of Rabinowitz' Global Bifurcation Theorem from Theorem II.3.3~\cite{Kielhoefer}. This fact will be proved in
  Appendix~B, see Lemma~\ref{lem:counterexample}.
   
  \medskip
    
  Both conditions (A1),(A2) are satisfied in the prototypical situation $F(x,\lambda)=x-\lambda
  Lx - H(x,\lambda)$ described above with the trivial solution family $\mathcal
  T=\{(0,\lambda):\lambda\in\R\}$. Usually, the study of bifurcations from non-standard trivial solution
  families $\mathcal T$, say $\mathcal T=\{(\bar x(t),\bar\lambda(t)):t\in\R\}$, is reduced to the case
  $\mathcal T=\{(0,\lambda):\lambda\in\R\}$ by considering the map $\tilde F(y,t):=F(\bar x(t)+y,\bar
  \lambda(t))$.
  Let us explain why we do not take this approach. Firstly, if $\bar\lambda(t)$ remains bounded as $t\to -\infty$ or $t\to
  \infty$, solutions of $F$ with parameter values outside $\{\bar\lambda(t):t\in\R\}$ cannot be described 
  by any result for the function $\tilde F$. Secondly, unbounded sequences of zeros of $\tilde F$
  need not correspond to unbounded sequences of zeros of $F$. Therefore, it is not possible to derive
  Rabinowitz' alternative from the corresponding result for $\tilde F$. Thirdly, if $\bar\lambda$ is not
  monotone, i.e. if $\mathcal T$ has turning points, then global continua of zeros of $\tilde F$ with respect
  to $(y,t)$ may be much more complicated than the ones for $F$ with respect to $(x,\lambda)$. One simple
  example for this is illustrated in Figure~\ref{fig:easy_example}. One finds that turning points of
  $\mathcal T$ become (artificial) bifurcation points with respect to the $(y,t)$-variables, which makes it rather complicated
  to establish Dancer's result about the jumps of the Leray-Schauder indices in this setting, especially when
  the number of bifurcation points is large.
  Finally, let us mention that bifurcations from such non--standard trivial solution families naturally appear
  in applications, see e.g. Section~\ref{sec:Applications} or~\cite{BaTiWa:Bifurcations,BaDaWa:bifurcations}
  for an application to a nonlinear elliptic Schr\"odinger system.

  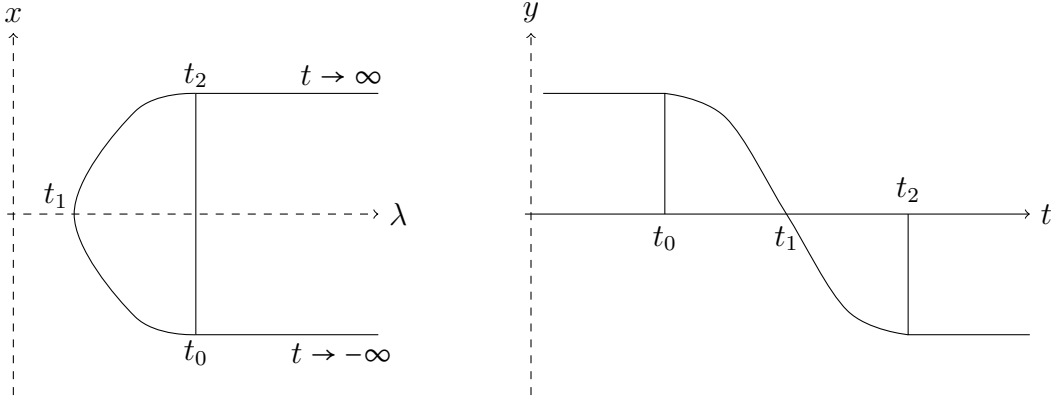
\begin{figure}[h!] 
    \centering
     \subfigure 
	{
    \begin{tikzpicture}[domain=-2:5, yscale=0.8, xscale=0.8]
	  \draw[->,dashed] (-3.1,0) -- (3,0) node[right] {$\lambda$};
	  \draw[->,dashed] (-3,-3) -- (-3,3) node[above]{$x$};
	  \draw[thin] (0,2)  -- (3,2);
	  \draw[thin] (0,-2) -- (3,-2);
	  \draw[thin] (0,-2) -- (0,2);
	  \node at (0,2.3){$t_2$};
	  \node at (2.4,2.3){$t\to\infty$};
	  \node at (0,-2.3){$t_0$};
	  \node at (-2.3,0.3){$t_1$};
	  \node at (2.4,-2.3){$t\to -\infty$};
   	  \draw plot[smooth, tension=0.6] coordinates {(0,2) (-1,1.7) (-2,0) (-1,-1.7) (0,-2)};
 	\end{tikzpicture}
 	}
 	\qquad
 	 \subfigure 
	{
	  \begin{tikzpicture}[domain=-0.85:1, yscale=0.8, xscale=0.8]
	  \draw[->] (-4.3,0) -- (4,0) node[right]{$t$};
 	  \draw[->,dashed] (-4.2,-3) -- (-4.2,3) node[above]{$y$};
 	  \node at (-2,-0.4){$t_0$};
	  \node at (0,-0.4){$t_1$};
	  \node at (2, 0.4){$t_2$};
 	  \draw[thin]  (-4,2) -- (-2,2) -- (-2,0);
 	  \draw[thin]  (4,-2) -- (2,-2) -- (2,0);
 	  \draw plot[smooth, tension=0.6] coordinates {(-2,2) (-1,1.6) (0,0) (1,-1.6) (2,-2)};
 	\end{tikzpicture}
	}
 	\caption{Bifurcation diagrams for $F$ and $\tilde F$, respectively}
 	\label{fig:easy_example}    
  \end{figure}

  \medskip
  
  For the statement of Rabinowitz' and Dancer's results under the relaxed assumptions (A1),(A2)
  we need  
  \begin{equation*} 
    \Sigma:=\{(x,\lambda)\in X\times\R: F(x,\lambda)=0\},\qquad
    \mathcal S:= \ov{\Sigma \sm \mathcal T}.
  \end{equation*}
  The index jump along the trivial solution family $\mathcal T$ in direction $\xi\in X'$ 
  at a bifurcation point $(x_0,\lambda_0)\in\mathcal T$ is defined by the formula 
  \begin{align}  \label{eq:formula1_deltastar}
    \begin{aligned}
    \delta^*(x_0,\lambda_0;\xi) 
    &:= \lim_{\mathcal T\ni (x,\lambda)\to (x_0,\lambda_0),  \atop 
    \lambda-\skp{\xi}{x-x_0}_{X'}>\lambda_0}
    \ind\big(F_x(x,\lambda)+F_\lambda(x,\lambda)\xi,0\big) \\ 
    &\quad -  \lim_{\mathcal T \ni (x,\lambda)\to (x_0,\lambda_0), \atop 
    \lambda-\skp{\xi}{x-x_0}_{X'}<\lambda_0} \ind\big(F_x(x,\lambda)+F_\lambda(x,\lambda)\xi,0\big).
    \end{aligned}
  \end{align}
  whenever these limits exist, i.e. whenever the involved Leray-Schauder indices are well-defined and
  eventually constant. Here, $\skp{\cdot}{\cdot}_{X'}$ denotes the dual pairing and
  $\ind(I-L,0)\in\{-1,+1\}$ is the Leray-Schauder index of  $I-L$ whenever $L$ is a compact linear operator
  with $1\notin \sigma(L)$. 
  In the classical setting $F(x,\lambda)=x-\lambda Lx-H(x,\lambda)$ and $\mathcal
  T=\{(0,\lambda):\lambda\in \R\}$ the number $\delta^*(0,\lambda_0;0)$ is well-defined 
  and equals $\sign(\lambda_0)n_L(\lambda_0)$ from Theorem~1 in \cite{Dancer_On_the_structure}. If,
  however, the bifurcation point $(x_0,\lambda_0)\in\mathcal T$ is also a turning point of $\mathcal T$, then
  $\xi=0$ is not admissible, since it is impossible to find solutions on $\mathcal T$ 
  on both sides of $\lambda_0$. So $\delta^*(x_0,\lambda_0;0)$ is not well-defined in this case. Instead of
  adding the unnatural assumption that bifurcation from turning points of $\mathcal T$ does not occur, we
  will therefore consider $\delta^*(x_0,\lambda;\xi)$ also for $\xi\neq 0$.
  Notice that the case $\xi\neq 0$ may be reduced to the case $\xi=0$ by a simple linear change of
  coordinates, see~\eqref{eq:change_coordinates}.

  \medskip
  
  In order to have $\delta^*(x_0,\lambda_0;\xi)$ well-defined, the direction $\xi\in X'$ has
  to be chosen in dependence of the trivial solution family $\mathcal T$. We say that $\xi\in X'$ is
  transverse to a subset of $\mathcal T$ if for each of its elements $(x_0,\lambda_0)$  a local
  $C^1-$parametrization $(\bar x,\bar\lambda)$ of $\mathcal T$ satisfies $(\bar
  x(t_0),\bar\lambda(t_0))=(x_0,\lambda_0)$ with $\bar\lambda'(t_0)-\skp{\xi}{\bar x'(t_0)}_{X'}\neq 0$. 
  In this case \eqref{eq:formula1_deltastar} gives the formula
  \begin{align} \label{eq:formula2_deltastar}
    \begin{aligned}
    \delta^*(x_0,\lambda_0;\xi) 
    &= \sign\big(\bar\lambda'(t_0)-\skp{\xi}{\bar
      x'(t_0)}_{X'} \big) \cdot  \\ 
    &\quad \Big[\; \lim_{t\to t_0^+}\ind\big(F_x(\bar x(t),\bar \lambda(t))+F_\lambda(\bar x(t),\bar
    \lambda(t))\xi,0\big)  \\
    &\quad -  \lim_{t\to t_0^-} \ind\big(F_x(\bar x(t),\bar \lambda(t))+F_\lambda(\bar x(t),\bar
    \lambda(t))\xi,0\big) \Big],
    \end{aligned}
  \end{align}
  which is useful in applications as we will see in Section~\ref{sec:Applications}. Notice that 
  transverse directions $\xi$ to any given finite subset of $\mathcal T$ always exist, which is a consequence
  of the Hahn-Banach Theorem.
  In the following theorem we summarize Rabinowitz' and Dancer's achievements in this general setting and we
  refer to Appendix~A for a proof. 
   
  \begin{thm}[Rabinowitz, Dancer] \label{thm:DanRab}
    Assume (A1),(A2) and $(x_0,\lambda_0)\in\mathcal S\cap \mathcal T$. Then the connected component $\mathcal
    C$ of $(x_0,\lambda_0)$ in $\mathcal S$ is either unbounded or it is bounded and satisfies  
    \begin{equation} \label{eq:degree_balance_RabDan}
      \sum_{(x,\lambda)\in\mathcal C\cap\mathcal T} \delta^*(x,\lambda;\xi)
      = 0
    \end{equation}
    whenever $\xi\in X'$ is transverse to each point in $\mathcal C\cap\mathcal T$.
  \end{thm}
  
  The condition~\eqref{eq:degree_balance_RabDan} 
  implies that there is an even number of points $(x,\lambda)\in\mathcal C\cap\mathcal T$ where 
  the critical eigenvalue of the operator $F_x(x,\lambda)+F_\lambda(x,\lambda)\xi$ 
  has odd algebraic multiplicity and crosses zero. 
  Notice that in Rabinowitz' and Dancer's original version for the special case
  $F(x,\lambda)=x-\lambda Lx - H(x,\lambda)$, the point $(x_0,\lambda_0)=(0,\lambda_0)$ is chosen in such a
  way that $1/\lambda_0$ is an eigenvalue of odd algebraic multiplicity of $L$ so that the above observation
  proves the existence of a second element of $\mathcal C\cap\mathcal T$. In this way, 
  \eqref{eq:degree_balance_RabDan} implies Rabinowitz' alternative. In the
  following, a continuum $\mathcal C$ as in Theorem~\ref{thm:DanRab} will be called a ''Rabinowitz continuum''
  of a given point $(x_0,\lambda_0)\in\mathcal T$.
  
  \section{Symmetry breaking and secondary bifurcations} \label{sec:symbreak}
  
  From now on we assume that there is a closed nontrivial subspace $Y\subsetneq X$ such that
  $F(Y\times\R)\subset Y$ for $F$ as in (A1). Moreover, the trivial solution family $\mathcal{T}$ will be
  assumed to belong to both spaces.   
  This makes sure that Theorem~\ref{thm:DanRab} is applicable both in $X\times\R$ and in $Y\times\R$. 
  Given that the relevant quantities introduced above in general depend on the ambient Banach space, we will
  put a corresponding index. For instance $\mathcal C_X,\mathcal C_Y$ will denote Rabinowitz continua in the
  spaces $X,Y$, respectively, similar for $\delta_X^*,\delta_Y^*,\Sigma_X,\Sigma_Y$ etc. We will actually see
  that the discrepancy between $\delta_X^*$ and $\delta_Y^*$  is responsible for symmetry
  breaking via global secondary bifurcation. Here the word symmetry breaking is justified since reasonable choices in
  applications are given by $Y=\{ x\in X: gx=x \,\text{for all }g\in G\}$ for some group $G$ acting linearly
  and continuously on $X$. For instance, in Section~\ref{sec:Applications} we will consider an ODE boundary
  value problem formulated in the spaces $X=H^1_{per}(2\pi/p;\C)$ and $Y=H^1_{per}(2\pi/q;\C)$ for $q=lp$ with $l\in\N$.
  Notice that this subspace $Y$ can indeed be rewritten as the fixed point space of the nontrivial group
  action $(\ov m x)(t)=x(t+2\pi m/q)$ for $x\in X$ and $\bar m\in G:=\Z/l\Z=\{\ov 0,\ldots,\ov{l-1}\}$.
  
  \medskip
  
  We will say that a continuum (i.e. a closed connected subset of $X\times\R$) $\mathcal{C}\subset \Sigma_X$
  bifurcates from $\mathcal{T}$ at the point $(x_0,\lambda_0)\in\mathcal T$ if $(x_0,\lambda_0)\in
  \ov{\mathcal{C}\sm\mathcal{T}}$.
  In such a situation we say that (local) secondary bifurcation occurs with respect to
  $(\mathcal{T},\mathcal{C})$ if there is a solution $(x,\lambda)\in \mathcal{C}\sm \mathcal{T}$ and a sequence of solutions
  $(x_n,\lambda_n)\in \Sigma_X\sm \mathcal{C}$ such that $(x_n,\lambda_n)\to (x,\lambda)$
  as $n\to\infty$. We then say that $(X,Y)$-symmetry breaking occurs at $(x,\lambda)$ if we can ensure  
    $\mathcal C\subset \Sigma_Y$ and $(x_n,\lambda_n)\in \Sigma_X \sm \Sigma_Y$, so  
  the $x_n$ are less symmetric than $x$. The secondary bifurcation will be called global if
  the connected component of $(x,\lambda)$ in $\ov{\Sigma_X\sm (\mathcal C\cup\mathcal T)}$ is unbounded or
  returns to the trivial family $\mathcal T$ at some other point on the trivial line, i.e., at some element of
  $\mathcal T\sm \mathcal C$. Another reasonable notion of global secondary bifurcation could require this
  connected component to be unbounded or to return to the larger set $\mathcal C\cup\mathcal T$ at another
  point. Our preference for the former definition is exclusively motivated by the fact that our main result
  from Theorem~\ref{thm_symbreak} allows to observe the former (more special) phenomenon.
  
  \medskip
  
  Local secondary bifurcation from $(\mathcal T,\mathcal C)$ is nothing but local
  bifurcation from $\mathcal C\sm\mathcal T$, so it is not a new concept from a theoretical point of
  view. Practically, however, this difference is huge, since $\mathcal C$ is rarely explicitly known so that
  standard bifurcation theorems are not applicable. In particular, the well-known tools for proving
  local bifurcations such as the Crandall-Rabinowitz Theorem \cite{CR_bif_simple} or the Marino-B\"ohme
  Theorem on variational bifurcation \cite{Boehme,Marino} are useless for studying such bifurcations. Degree
  theory as used in Theorem~\ref{thm:DanRab}, however, allows for global considerations and turns out to be
  useful. The following lemma shows how this theorem may be employed to prove global secondary bifurcation on
  an abstract level.
  
  \begin{lem} \label{lem:bifurcation}
    Let $X$ be a real Banach space, $Y\subset X$ a closed subspace and assume (A1),(A2) as well as
    $F(Y\times\R)\subset Y$.
    Suppose that the Rabinowitz continuum $\mathcal C_Y$ emanating from $(x_0,\lambda_0)\in\mathcal T\subset
    Y\times\R$ is non-empty and bounded in $Y\times\R$ and satisfies
    \begin{itemize}
      \item[(a)] $\sum_{(x,\lambda)\in \mathcal{C}_Y\cap\mathcal{T}} \delta_X^*(x,\lambda;\xi) \neq 0$,
      \item[(b)] $\mathcal C_X\cap U = \mathcal C_Y\cap U$ for some open neighbourhood $U\subset X\times\R$ of
      $\mathcal{C}_Y\cap\mathcal{T}$
    \end{itemize}
    for some direction $\xi\in X'$ that is transverse to $\mathcal C_X\cap \mathcal T$.
    Then the following alternative holds for the Rabinowitz continuum $\mathcal C_X$ emanating from
    $(x_0,\lambda_0)$:
    \begin{itemize}
      \item[(i)] $\mathcal{C}_X$ is unbounded or 
      \item[(ii)] $\mathcal{C}_X\cap\mathcal{T} \supsetneq \mathcal{C}_Y\cap\mathcal{T}$ such that
         $\sum_{(x,\lambda)\in \mathcal{C}_X\cap\mathcal{T}} \delta_X^*(x,\lambda;\xi) = 0$
         for $\xi\in X'$ transverse to $\mathcal C_X\cap \mathcal T$.
    \end{itemize}
    In both cases global secondary bifurcation occurs from $(\mathcal T,\mathcal C_Y)$ through $\mathcal C_X$
    and $(X,Y)$-symmetry-breaking occurs at all points of $\ov{\mathcal C_X\sm\mathcal C_Y}\cap (\mathcal
    C_Y\sm\mathcal T)\neq \emptyset$.
  \end{lem}
  \begin{proof} 
    From  $Y\subset X$ 
    we get $\mathcal C_Y\subset\mathcal C_X$. If now $\mathcal C_X$ is bounded, then
    Theorem~\ref{thm:DanRab}~(ii) yields 
    $$
        \sum_{(x,\lambda)\in \mathcal{C}_X\cap\mathcal{T}} \delta_X^*(x,\lambda;\xi) = 0,
    $$
    so assumption (a) implies $\mathcal{C}_X\cap\mathcal{T}\supsetneq \mathcal
    C_Y\cap \mathcal T$. This proves the alternative (i) or (ii) from above. Next we use (b) to
    prove global secondary bifurcation from $(\mathcal T,\mathcal C_Y)$ through $\mathcal C_X$. 
    The set $\mathcal C_X$ is, by definition, connected in
    $\ov{\Sigma_X\sm\mathcal T}$ and we have $\mathcal C_X = \ov{\mathcal C_X \sm \mathcal C_Y}\cup
    \mathcal{C}_Y$ where both subsets are nonempty and closed in $\ov{\Sigma_X\sm\mathcal T}$. So
    these two sets have nonempty intersection, i.e., we can find $(x,\lambda)\in\mathcal C_Y$ and
    $(x_n,\lambda_n)\in\mathcal{C}_X\sm \mathcal C_Y$ such that $(x_n,\lambda_n)\to (x,\lambda)$ as
    $n\to\infty$. By assumption (b) the continua $\mathcal C_X,\mathcal C_Y$ coincide in a neighbourhood of
    $\mathcal{C}_Y\cap \mathcal{T}$ proving
    $(x,\lambda)\in\mathcal C_Y\sm \mathcal T$, i.e., local secondary bifurcation with respect to
    $(\mathcal T,\mathcal C_Y)$ occurs at $(x,\lambda)$. Even more, for any such
    point $(x,\lambda)\in \ov{\mathcal C_X \sm \mathcal C_Y}\cap \mathcal C_Y\sm \mathcal T$ and any open
    neighbourhood of $(x,\lambda)$ there must be at least one element of $\mathcal C_X\sm\mathcal C_Y$ that
    does not belong to $Y\times\R$, since otherwise $\mathcal C_Y$ would not be maximal. So $(X,Y)$-symmetry
    breaking occurs at $(x,\lambda)$.
    
    \medskip
    
    We finally prove that the secondary bifurcation is global in the sense defined above. In view of the
    validity of the alternative ``(i) or (ii)'' it suffices to show that the connected component of any
    $(x,\lambda) \in\mathcal C_Y\sm\mathcal T$ in $\ov{\Sigma_X \sm (\mathcal C_Y\cup \mathcal T)}$ is precisely
    $\ov{\mathcal C_X\sm (\mathcal C_Y\cup \mathcal T)}$. Indeed, the latter set  contains $(x,\lambda)$ and it is closed in
    $\ov{\Sigma_X \sm (\mathcal C_Y\cup \mathcal T)}$. Additionally, it is open in $\ov{\Sigma_X \sm
    (\mathcal C_Y\cup \mathcal T)}$ since $\mathcal C_X$ is open in $\ov{\Sigma_X\sm \mathcal T}$. 
    This finishes the proof. 
  \end{proof}
  
  In Figure~\ref{fig:bifurcation_prim_sec} we illustrate the situation described by
  Lemma~\ref{lem:bifurcation} schematically. The curve of trivial solutions $\mathcal T$ contains primary
  bifurcation points $P_1,P_2,P_3,P_4$. One possible configuration is that that the primary branch $\mathcal
  C_Y$ consists of the solutions on the curves joining $P_1,S_1,P_4$. At $S_1$ secondary bifurcation
  occurs into $\mathcal C_X\supset \mathcal C_Y$ and $\mathcal C_X$ reenters the trivial solution family at
  $P_2$ and $P_3$.
  
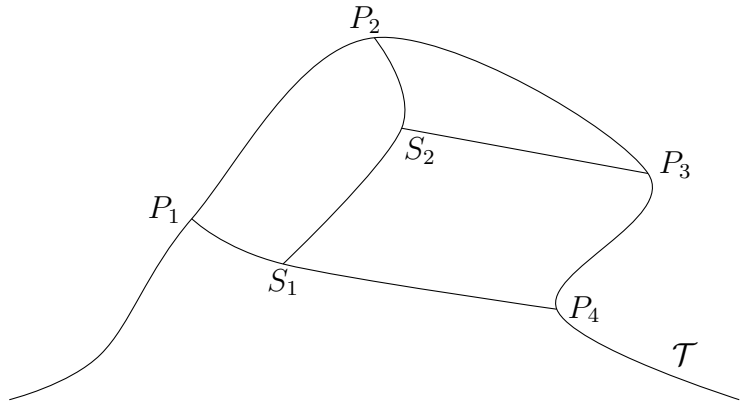
\begin{figure}[h!]
    \begin{tikzpicture}[domain=0:10, yscale=0.6, xscale=1.2]
	  \node at (7.4,1){$\mathcal T$};
	  \node at (1.7,4.2){$P_1$};
	  \node at (3.9,8.4){$P_2$};
	  \node at (7.3,5.2){$P_3$};
	  \node at (6.3,2){$P_4$};
	  \node at (3,2.6){$S_1$};
	  \node at (4.5,5.5){$S_2$};
   	  \draw plot[smooth, tension=0.6] coordinates {(0,0) (1,1) (2,4) (4,8) (7,5) (6,2) (8,0)};
   	  \draw plot[smooth, tension=0.6] coordinates {(2,4) (3,3) (6,2)};
   	  \draw plot[smooth, tension=0.6] coordinates {(3,3) (4.3,6) (4,8)};
   	  \draw plot[smooth, tension=0.6] coordinates { (4.3,6) (7,5)};
 	\end{tikzpicture}
 	\caption{Primary and secondary bifurcations at $P_1,P_2,P_3,P_4$ resp. $S_1,S_2$}
 	\label{fig:bifurcation_prim_sec}
\end{figure}

  At first sight, Lemma~\ref{lem:bifurcation} may appear to be of limited use due to assumption (b).
  In order to verify it, we make use of the Crandall-Rabinowitz Theorem because it allows to
  charaterize all solutions in the vicinity of the bifurcation point. For the convenience of the reader we
  recall it here.
  
   \begin{thm}[Crandall-Rabinowitz, cf. Theorem~1 \cite{CR_bif_simple}] \label{thm_CR}
     Let $X$ be a real Banach space and assume (A1),(A2) as well as $F\in C^2(X\times\R,X)$, let
     $(x_0,\lambda_0):=(\bar x(t_0),\bar\lambda(t_0))\in\mathcal T$ where  $(\bar x,\bar\lambda)$ is a local $C^1-$parametrization of
     $\mathcal T$. Moreover assume
    \begin{itemize}
      \item[(A3)] $\ker_{X\times\R}(F'(x_0,\lambda_0))=\spa\{(\bar x'(t_0),\bar\lambda'(t_0)),\phi\}$ is
      two-dimensional with $\phi\in Y\times\R$,
      \item[(A4)] $\ran_{X\times\R}(F'(x_0,\lambda_0))$ has codimension one, 
      \item[(A5)] $F''(x_0,\lambda_0)[(\bar x'(t_0),\bar\lambda'(t_0)),\phi]\notin
      \ran_{X\times\R}(F'(x_0,\lambda_0))$.
    \end{itemize}
    Then there exists $\eps>0$ and a continuous curve $(\hat x,\hat \lambda):
    (-\eps,\eps)\to X\times \R$ with $\hat x(0)=x_0,\hat\lambda(0)=\lambda_0$ such that 
    $F(\hat x(s),\hat\lambda(s))=0$ for all $s\in (-\eps,\eps)$ and $(\hat x(s),\hat
    \lambda(s))\notin\mathcal T$ if $0<|s|<\eps$. In a small neighbourhood of $(x_0,\lambda_0)$ in
    $X\times\R$ all solutions not belonging to $\mathcal T$ lie on this curve. 
  \end{thm} 
  
  We remark that the regularity assumptions on $F$ may be slightly relaxed for this result to remain true. In
  fact, twice continuous differentiablity on $X\times\R$ may be replaced by once continuous differentiablity
  in a neighbourhood of $\mathcal T$ with uniform twice continuous differentiablity along $\mathcal T$, see
  Satz~A.7 and in particular assumption (V) in~\cite{Ma_Thesis}, pp. 119--126. Combining the
  Crandall-Rabinowitz Theorem with Lemma~\ref{lem:bifurcation} we obtain our main result.
   
  \begin{thm}\label{thm_symbreak}
    Let $X$ be a real Banach space, $Y\subset X$ a closed subspace and assume (A1),(A2) as well as $F\in
    C^2(X\times\R,X)$ with $F(Y\times\R)\subset Y$. Moreover suppose that the Rabinowitz continuum
    $\mathcal{C}_Y$ emanating from $(x_0,\lambda_0)\in\mathcal T\subset Y\times\R$ is non-empty, bounded and
    that (A3),(A4),(A5) are satisfied at each element of $\mathcal C_Y\cap\mathcal T$.
    Furthermore, for $\xi\in X'$ transverse to $\mathcal C_X\cap\mathcal T$ we assume 
    \begin{equation}\label{eq:sym_break_condition} 
      \sum_{(x,\lambda)\in \mathcal{C}_Y\cap \mathcal{T}}  \delta_X^*(x,\lambda;\xi)\neq 0.
    \end{equation}
    Then global secondary bifurcation occurs from $(\mathcal{T},\mathcal{C}_Y)$ through $\mathcal C_X$ 
    and $(X,Y)$-symmetry-breaking occurs at all points of $\ov{\mathcal C_X\sm\mathcal C_Y}\cap (\mathcal
    C_Y\sm\mathcal T)\neq \emptyset$. 
  \end{thm}
  \begin{proof}
     We verify the assumptions of Lemma~\ref{lem:bifurcation}. Condition (a) holds thanks
     to~\eqref{eq:sym_break_condition}. In order to prove (b) we may write $\mathcal{C}_Y\cap \mathcal{T} =
     \{(x_0,\lambda_0),\ldots,(x_k,\lambda_k)\}$ thanks to discreteness of the set of bifurcation points on
     $\mathcal T$ from Proposition~\ref{prop:unif_diff}.
     By (A1)-(A5) the Crandall-Rabinowitz Theorem is applicable at each $(x_j,\lambda_j)$ both in $X$ and in
     $Y$, so we obtain continuous curves $(\hat x_j,\hat\lambda_j):(-\eps_j,\eps_j)\to Y\times\R$ and 
     w.l.o.g. mutually disjoint small neighbourhoods $U_j$ of $(x_j,\lambda_j)$ in $X\times\R$ with the
     properties mentioned in Theorem~\ref{thm_CR}. So $U:= \bigcup_{j=0}^k U_j$ is a neighbourhood of
     $\mathcal C_Y\cap \mathcal T$ in $X\times\R$ with the property 
     $$ 
       \mathcal C_Y \cap U 
       = \mathcal C_X \cap U 
       = \bigcup_{j=0}^k \{ (\hat x_j(s),\hat\lambda_j(s)) : |s|<\eps_j\} \cap U_j. 
     $$
     This proves (b) so that Lemma~\ref{lem:bifurcation} gives the result.    
 \end{proof}

  \section{Applications} \label{sec:Applications}

  In this section we apply Theorem~\ref{thm_symbreak} in order to detect secondary bifurcations via
  period-doubling, period-tripling, etc. for the stationary Lugiato-Lefever equation    
  \begin{equation} \label{eq:LL}
    d a'' + (i-\zeta)a + |a|^2a - if = 0,\qquad a:\R\to\C\text{ is }2\pi-\text{periodic}.  
  \end{equation}
  It was proposed in \cite{Lugiato_Lefever1987} as an accurate model for the description of the
  electric field in a ring resonator. The parameters $d,\zeta,f\in\R$ with $d,f\neq 0$ model
  physical effects originating from dispersion, detuning and forcing, respectively, and therefore vary
  according to the precise experimental setup. The term $ia$ incorporates damping and the nonlinear term is
  due to the use of Kerr-type materials as propagation media. The parameters $f$ and especially $\zeta$ may be
  calibrated rather easily in the laboratory in order to generate so-called frequency combs, i.e., electric
  fields with a broad frequency range of almost uniformly distributed and sufficiently large power per
  frequency. Such electric fields typically arise as spatially concentrated (soliton-like) solutions of \eqref{eq:LL}. In a joint work with
  {W. Reichel} \cite{MaRe} the author provided a detailed bifurcation analysis of the Lugiato-Lefever equation
  related to primary bifurcations from the family of constant solutions. Our aim here is to discuss secondary
  bifurcations for this problem. To this end we will show how the assumptions of Theorem~\ref{thm_symbreak}
  may be verified in the context of \eqref{eq:LL}. In order not to overload this paper with tedious
  computations, we will only present the main steps.
  We start by  recalling the results obtained in~\cite{MaRe} about the primary bifurcations. 
    
  \subsection*{The functional analytical setting}  

  In order to prove the existence of nonconstant solutions via bifurcation theory it was
  shown in~\cite{MaRe} that so-called synchronized solutions of~\eqref{eq:LL} are precisely the zeros of the
  function $F:H^1_{per}([0,2\pi];\C)\times\R \to H^1_{per}([0,2\pi];\C)$ given by
  \begin{equation*}
	F(a,\zeta) :=  a- \sign(d) D^{-1}(-\zeta a + \sign(d) a+ |a|^2a + i a -if),
  \end{equation*}
  where $D$ denotes the differential operator $-|d|\frac{d^2}{dx^2}+1$ with homogeneous Neumann boundary
  conditions at $0$ and $\pi$ for both real and imaginary part, see Section~4.1 and equation~(23)
  in~\cite{MaRe} for details. These boundary conditions were
  chosen in order to benefit from simple kernels by ruling out the translation invariance of \eqref{eq:LL}. Moreover,
  they ensure the solutions to be symmetric about $0$ and $\pi$ and hence to be $2\pi$-periodic. Since not all
  solutions of~\eqref{eq:LL} are known to satisfy these boundary conditions, this special class of solutions
  was attributed the name ''synchronized'', see Definition~1.4 in~\cite{MaRe}.
  So $F$ satisfies assumption (A1), the parameters $d\neq 0,f\in\R$ are fixed and $\zeta$ will
  be considered as a bifurcation parameter.

  \subsection*{The trivial solution family $\mathcal T$ and its primary bifurcations} 

  In Lemma~2.1~(a) from \cite{MaRe} it was proved that there is a uniquely determined
  (unbounded) curve $\mathcal T = \{ (\bar a(t),\bar \zeta(t)):|t|<1\}$ consisting of constant solutions
  of~\eqref{eq:LL}. This curve is smoothly parametrized via
  \begin{equation}\label{eq:trivial_curve_parametrization}
     \bar a(t) = f(1-t^2)-ift(1-t^2)^{1/2},\qquad \bar\zeta(t) = f^2(1-t^2)+ t(1-t^2)^{-1/2} \qquad (|t|<1).  
  \end{equation} 
  Moreover, it was shown that for ``generic'' choices of $d$ and $f$ there are finitely many bifurcation
  points on $\mathcal T$ at $t=t_{k,1}$ or $t=t_{k,2}$ for $k=1,\ldots,k_{max}$. These points are
  characterized as the solutions of
  \begin{equation} \label{eq:FC_bifurcation_condition}
    (\bar\zeta(t)+dk^2)^2-4|\bar a(t)|^2(\bar\zeta(t)+dk^2) + 1+ 3|\bar a(t)|^4 = 0 \qquad
    (k=1,\ldots,k_{max}),
  \end{equation}
  see Proposition~4.3~\cite{MaRe}. In particular, (A2) is satisfied. In a neighbourhood of these bifurcation
  points the associated Rabinowitz continua consist of $2\pi/k$-periodic solutions, and they are bounded and
  therefore return to $\mathcal T$ at some other bifurcation point. This analysis benefits from the fact that
  the sufficient conditions of the Crandall-Rabinowitz Theorem (A3),(A4),(A5) are satisfied in all
  bifurcation points for ``generic'' $d$ and $f$. This follows from the fact that the assumptions
  (S),(T) from Theorem~1.4 in~\cite{MaRe} hold for most parameter values.
  The numerical investigations from Section~5.3 \cite{MaRe} suggest that the periodic pattern close to the bifurcation point may be lost
  along some of the bifurcating branches via secondary bifurcation. In the following we show how this
  phenomenon may be proved with the aid of Theorem~\ref{thm_symbreak}. Our theoretical results
  are illustrated in Figure~\ref{fig:2ndaryBif} using the Matlab software package \textit{pde2path}.
  
  \subsection*{Index computations}

  We apply Theorem~\ref{thm_symbreak} for the Banach spaces
   \begin{equation}\label{eq:XY}
    X= H^1_{per}\big(\frac{2\pi}{p};\C\big),\quad Y=H^1_{per}\big(\frac{2\pi}{q};\C\big) \qquad\text{where }p 
    \text{ divides }q. 
  \end{equation}
  Since we will not encounter bifurcation from turning points of $\mathcal T$ later, we only consider 
  $\xi=0$. For all bifurcation points $(\bar a(t_{k,i}),\bar\zeta(t_{k,i}))\in \mathcal
  C_Y\cap\mathcal T$ with $k\in q\N$ we will have to compute 
  \begin{equation}\label{eq:bifurcation_indices}
    \delta^*_X(\bar a(t_{k,i}),\bar\zeta(t_{k,i});0) 
    = \sign(\bar\zeta'(t_{k,i}))\cdot  (\iota_X(t_{k,i}+\eps)-\iota_X(t_j-\eps)) \qquad (i=1,2),
  \end{equation}
  see~\eqref{eq:formula2_deltastar}, where $\iota_X(t)$ is given by 
  $$
    \iota_X(t)
    :=\ind_X(F(\cdot,\bar\zeta(t)), \bar a(t))
    =\ind_X(F_a(\bar a(t),\bar\zeta(t)),0).
  $$ 
  Given the definition of the Leray-Schauder index and the fact that all eigenvalues are simple, the
  quantities $\iota_X(t_{k,i}+\eps),\iota_X(t_{k,i}-\eps)$ can be computed by counting the
  negative eigenvalues of the linearized operator $F_a(\bar a(t),\bar \zeta(t))[\cdot]:X\to X$. These
  eigenvalues can be computed with the aid of Proposition~4.3~\cite{MaRe}. In the notation from~\cite{MaRe}
  one finds that $E$ is an eigenvalue of this operator in $X$ if and only if the determinant of one the matrices
  $dl^2\Id-N(a,\zeta)-E(dl^2+\sign(d))\Id$ for $l\in p\N$ vanishes. Plugging in the formula for $N(a,\zeta)$
  from Proposition~4.2 \cite{MaRe} we get that the eigenvalues satisfy
  \begin{equation} \label{eq:eigenvalue_formula}
     (\bar\zeta(t)+dl^2-E(dl^2+\sign(d)))^2-4|\bar a(t)|^2(\bar\zeta(t)+dl^2-E(dl^2+\sign(d))) + 1 + 3|\bar
     a(t)|^4 = 0  
  \end{equation}
  for some $l\in p\N$ coming with $2\pi/l-$periodic eigenfunctions. So $\iota_X(t_{k,i}\pm \eps)$ is the
  number of negative $E$ solving \eqref{eq:eigenvalue_formula} for some $l\in p\N$, which can be computed
  rather easily using a computer.
  Notice that the eigenvalues change with $p$ and hence
  with the ambient space $X$.

  \subsection*{Computing $\mathcal C_Y\cap \mathcal T$}
  
  Above we pointed out that all bifurcation points in $H^1_{per}([0,2\pi];\C)$ are  
  of the form $(\bar a(t_{k,1}),\bar\zeta(t_{k,1})),(\bar a(t_{k,2}),\bar\zeta(t_{k,2}))$
  for $k=1,\ldots,k_{max}$ with $2\pi/k$-periodic eigenfunctions of the (simple) zero eigenvalue.
  Choosing $Y$ as in \eqref{eq:XY} we find that these points, being bifurcation points in $Y$, can  belong to
  $\mathcal C_Y$ only if $k\in q\N$. So the choice $k_{max}/2<q\leq k_{max}$ ensures $k=q$. In particular,
  $\mathcal C_Y\cap\mathcal T \subset \{ (\bar a(t_{q,1}),\bar\zeta(t_{q,1})), (\bar
    a(t_{q,2}),\bar\zeta(t_{q,2}))\}$.  
  Moreover, $\mathcal C_Y$ is bounded by the a priori bounds from Theorem~1.1 and Theorem~1.2 in~\cite{MaRe}, 
  so that the set $\mathcal C_Y\cap\mathcal T$ contains at least two elements by Rabinowitz' bifurcation
  theorem, see~\eqref{eq:degree_balance_RabDan} and the explanations thereafter. 
  From these two facts we deduce 
  \begin{equation} \label{eq:CYT}
    \mathcal C_Y \cap \mathcal T = \{ (\bar a(t_{q,1}),\bar\zeta(t_{q,1})), (\bar
    a(t_{q,2}),\bar\zeta(t_{q,2}))\} \qquad\text{provided }q\in\N,\,k_{max}/2<q\leq k_{max}.
  \end{equation}

  \subsection*{Summary}  
  For generic $d\neq 0$ and $f\in\R$ we find $k_{max}\in\N_0$ such that the curve of constant solutions
  $\mathcal T$ given by \eqref{eq:trivial_curve_parametrization} contains $2k_{max}$ bifurcation points in
  $H^1_{per}([0,2\pi];\C)$ at $t=t_{k,1}$ or $t=t_{k,2}$ for $k=1,\ldots,k_{max}$.
  Choosing then $X,Y$ as in \eqref{eq:XY} with $k_{max}/2<q\leq k$ we get \eqref{eq:CYT}. The
  sufficient condition~\eqref{eq:sym_break_condition} for symmetry breaking secondary bifurcation can then be 
  verified using~\eqref{eq:bifurcation_indices} where the Leray-Schauder indices $\iota_X(t_j\pm \eps)$ 
  is $-1$ to the number of negative $E$ solving \eqref{eq:eigenvalue_formula} for some $l\in p\N$. 

  \subsection*{Secondary bifurcations}
  
  For simplicity we now focus on a special case. We choose $f=1.6,d=0.1$ so that the interested reader may
  compare our results to those presented in Section~5.3 of the paper~\cite{MaRe}. 
  In this particular case the equation~\eqref{eq:FC_bifurcation_condition} has
  exactly 14 solutions, two for each $k\in\{1,\ldots,7\}$, i.e., $k_{max}=7$. The numerical values for these
  solutions $t_{k,1},t_{k,2}$ yielding the bifurcation points $(\bar a(t_{k,1}),\bar\zeta(t_{k,1})),(\bar
  a(t_{k,2}),\bar\zeta(t_{k,2}))$ in the ambient space $H^1_{per}([0,2\pi];\C)$ are provided
  in~Figure~\ref{bifurcation_points}.
  
  \medskip 
    
\begin{figure}[h!] 
  \renewcommand*{\arraystretch}{1.2}
\centering
\begin{tabular}{|c||c|c|c|c|c|c|c|} \hline 
~ & $k=1$ & $k=2$ & $k=3$ & $k=4$  & $k=5$ & $k=6$ & $k=7$  \\ 
\hline 
$t_{k,1}$ &  $~0.10528$   & $-0.18543$ & $-0.52046$  & $-0.72866$  &  $-0.77281$ &  $-0.6169$5 & $-0.20600$ \\
$t_{k,2}$  &   $~0.77130$ & $~0.75556$ &  $~0.72127$ &  $~0.66089$ & $~0.56321$ & $~0.40312$  & $~0.01535$ \\ 
\hline  
 $\bar\zeta(t_{k,1})$ &  $~2.63750$    &  $~2.28327$   & $~1.25702$  
 &  $~0.13682$ & $-0.18666$ & $~0.80166$ & $~2.24085$ \\
 $\bar\zeta(t_{k,2})$ & $~2.24888$  &   $~2.25196$ &  $~2.26952$
 &  $~2.32248$ & $~2.42954$ &  $~2.58449$ &  $~2.57475$ \\
 \hline  
 $\Real(\bar a(t_{k,1}))$ &  $~0.64816$   &
 $~0.68661$   & $~0.76763$  
 &  $~0.90117$ & $~1.09247$ & $~1.34000$ & $~1.59962$ \\
 $\Imag(\bar a(t_{k,1}))$ & $-0.78546$  & $-0.79192$ & $-0.79934$ & $-0.79358$ & $-0.74462$ &
 $-0.59026$ & $-0.02455$\\
 $\Real(\bar a(t_{k,2}))$ & $~1.58226$ &   $~1.54499$ &  $~1.16659$
 & $~0.75049$ & $~0.64442$ & $~0.99099$ &  $~1.53210$\\
 $\Imag(\bar a(t_{k,2}))$  & $-0.16752$ &  $~0.29154$
 & $~0.71106$ & $~0.79847$ &  $~0.78473$ & $~0.77687$ & $~0.32253$\\
\hline
\end{tabular}
  \caption{Bifurcation points on $\mathcal T$ for $f=1.6,d=0.1$}
  \label{bifurcation_points}
\end{figure} 

  \medskip
  
  For notational convenience we write 
  $z_{k,i}:=(\bar a(t_{k,i}),\bar\zeta(t_{k,i}))$ for the bifurcation points in
  $H^1_{per}([0,2\pi];\C)$. Using \eqref{eq:CYT} it is possible to check the symmetry breaking condition
  \eqref{eq:sym_break_condition} with the aid of formula~\eqref{eq:bifurcation_indices}. 
  Doing so for the spaces $X,Y$ from \eqref{eq:XY},
  Theorem~\ref{thm_symbreak} yields the following:
  \begin{itemize}
    \item[(1)] $q=7,p=1$: Then $\delta_X^*(z_{7,1})+\delta_X^*(z_{7,2})=-4$, so the
    symmetry-breaking condition \eqref{eq:sym_break_condition} from Theorem~\ref{thm_symbreak} is satisfied 
    and secondary bifurcation from $(\mathcal{T},\mathcal C_Y)$ occurs via period-septupling. 
    \item[(2)] $q=6,p=3$: Here we find $\delta_X^*(z_{6,1})+\delta_X^*(z_{6,2})=-4$, which implies
    secondary bifurcation by period-doubling from $2\pi/6-$periodic into $2\pi/3$-periodic solutions. 
    Moreover,  $\delta_X^*(z_{3,1})=\delta_X^*(z_{3,2})=2$ implies  
    $\mathcal C_X\cap \mathcal T = \{ z_{6,1}, z_{6,2}, z_{3,1}, z_{3,2}\}$ because of
    Lemma~\ref{lem:bifurcation}~(ii). Item (3) even reveals that in a larger space, for instance in
    $H^1_{per}([0,2\pi];\C)$, we will discover further secondary bifurcations. 
    \item[(3)] $q=6,p=2$: Then $\delta_X^*(z_{6,1})+\delta_X^*(z_{6,2})=4$, so secondary
    bifurcation via period-tripling occurs. 
    \item[(4)] $q=4,p=2$: From $\delta_X^*(z_{4,1})+\delta_X^*(z_{4,2})=-4$ we deduce secondary bifurcation
    via period-doubling. 
\end{itemize}
  In particular, we conclude: for $f=1.6,d=0.1$ and each pair $(q,p)\in
  \{(7,1),(6,3),(6,2),(4,2)\}$ there is a sequence of $2\pi/p$-symmetric but not $2\pi/q$-symmetric solutions
  of~\eqref{eq:LL} converging to a nonconstant $2\pi/q$-symmetric solution. 

  \medskip
   
  \begin{figure}[h!] 
   \subfigure[$q=7,p=1$]{ 
      \includegraphics[scale=0.53]{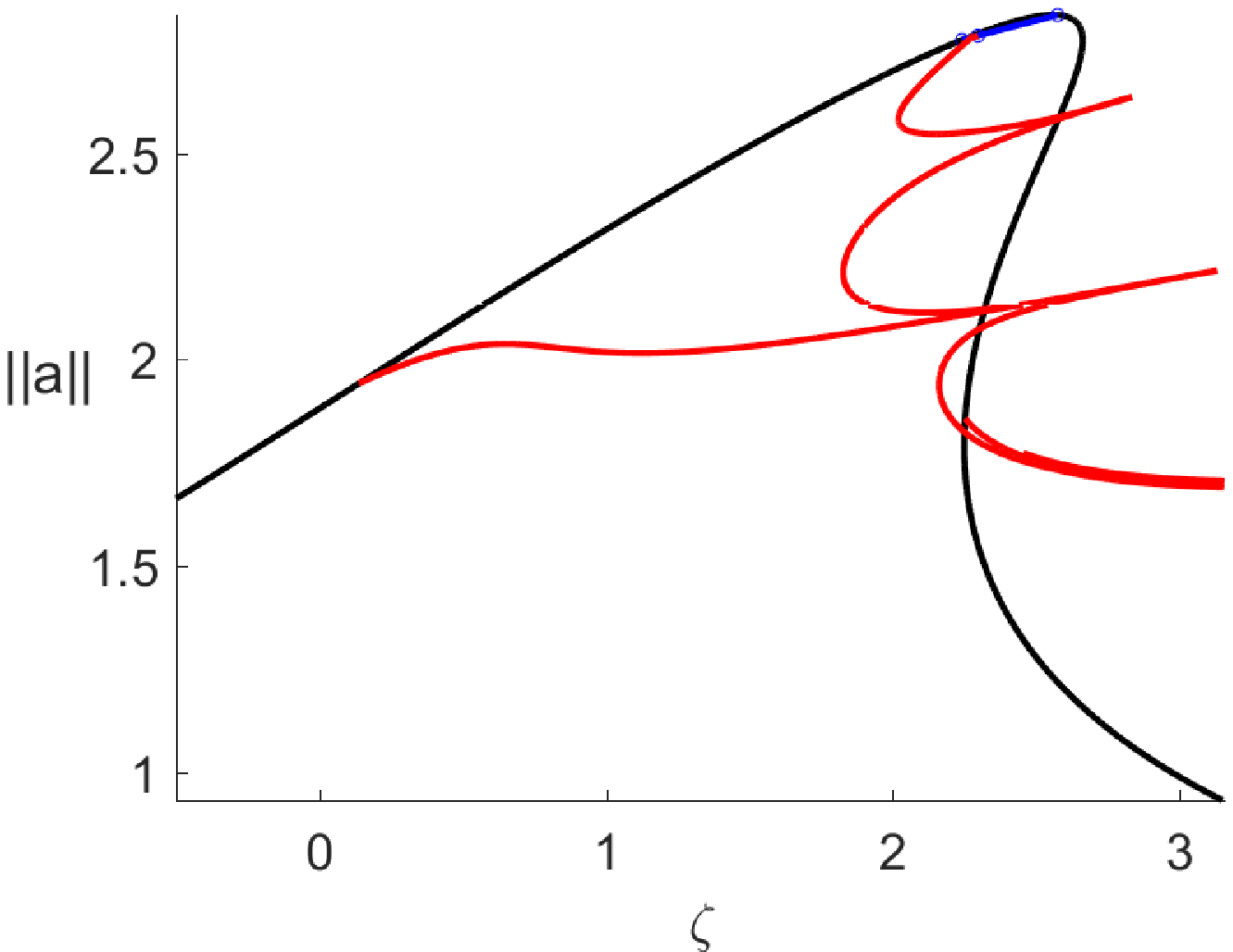}
      }
  \hfill
  \subfigure[$q=6,p=3$]{  
      \includegraphics[scale=0.53]{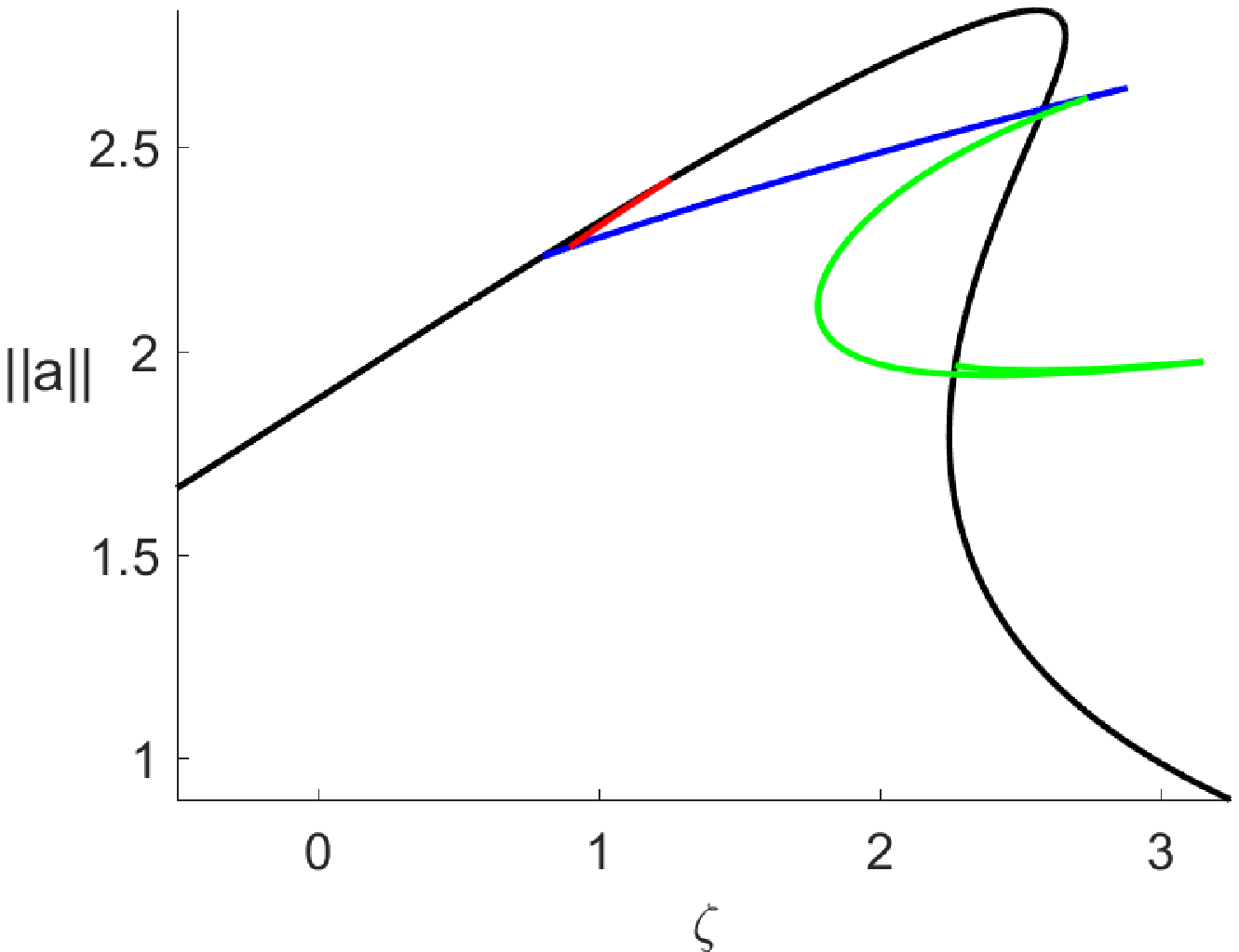}
      }
   \subfigure[$q=6,p=2$]{ 
      \includegraphics[scale=0.53]{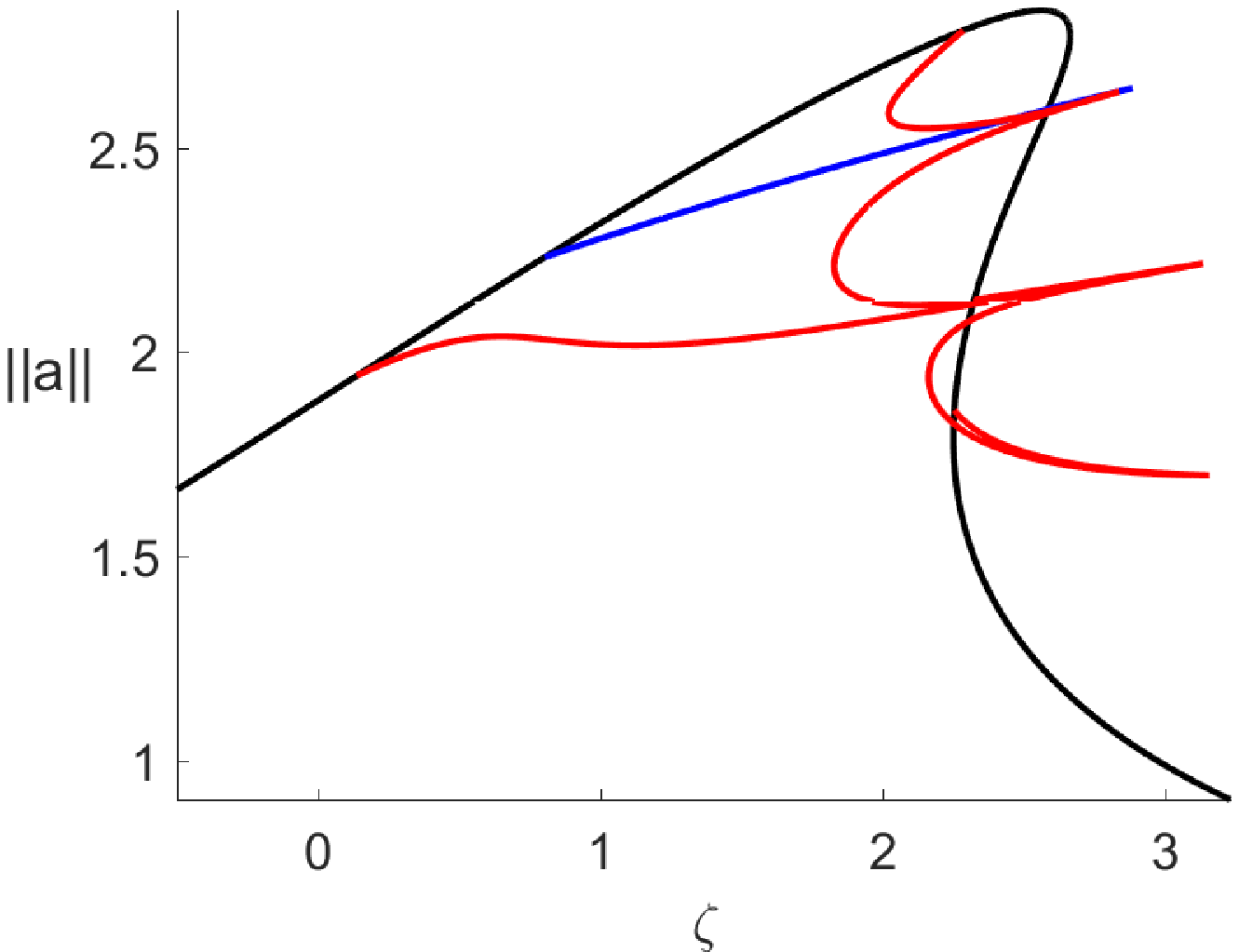}
      } 
  \hfill
  \subfigure[$q=4,p=2$]{ 
      \includegraphics[scale=0.53]{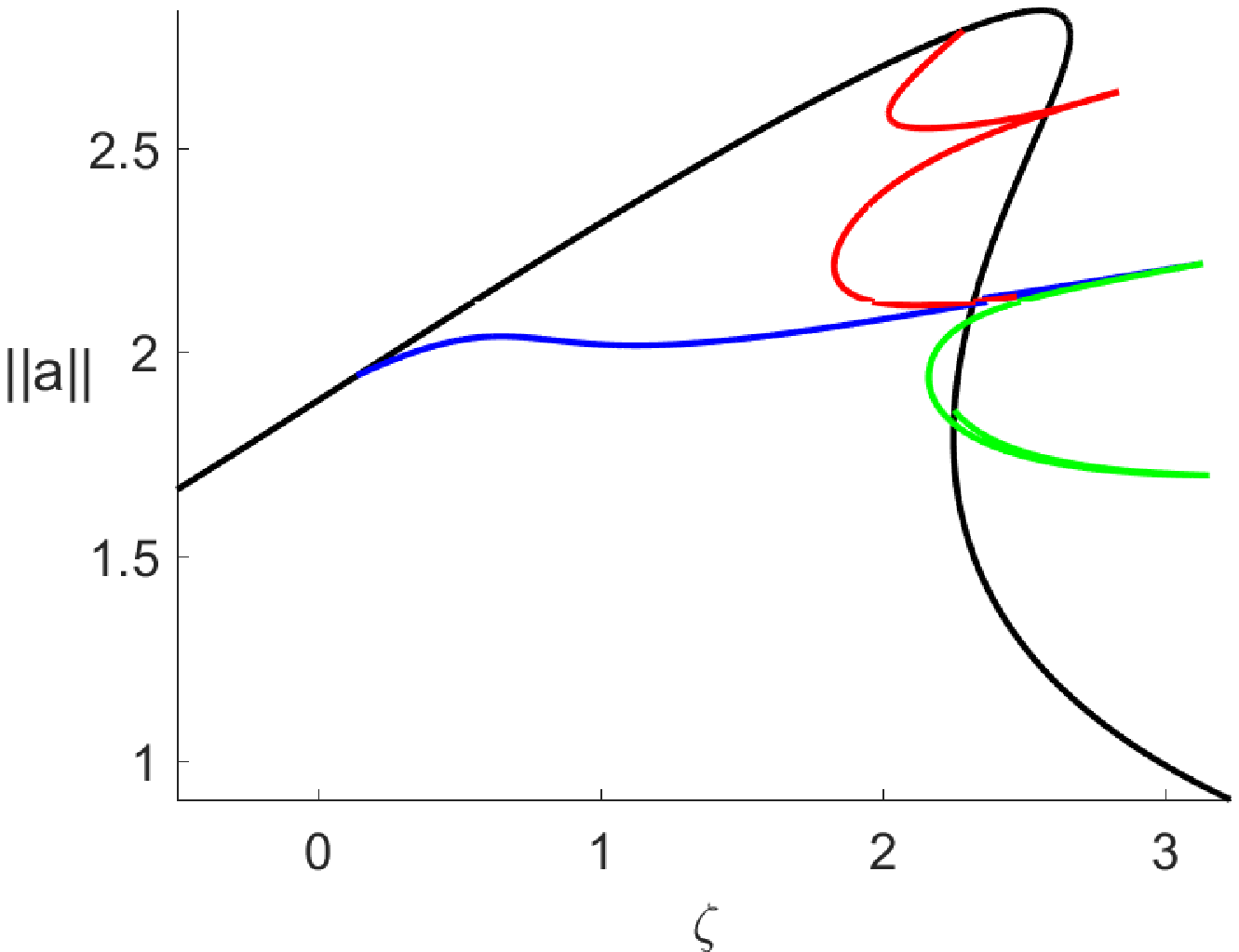}
      }
   \caption{$\mathcal T$ (black), primary branches (blue), secondary branches (green,red)}
   \label{fig:2ndaryBif}
\end{figure}
 
 \section{Appendix A: Proof of Theorem~\ref{thm:DanRab}}
   
   We finally provide the proof of Dancer's and Rabinowitz' results from Theorem~\ref{thm:DanRab}
   under the relaxed assumptions (A1),(A2). As in \cite{Dancer_On_the_structure,Rab:global} the main
   arguments rely on well-known properties of the Leray-Schauder degree. We refer to the survey article 
   \cite{Maw_LSdegree} and the books \cite{AmbMal:Nonlinear_analysis,Deim} for more information about these
   topics. The following property will be especially important.
    

  \begin{lem}[Generalized homotopy invariance] \label{lem:hom_inv}
    Let $X$ be a real Banach space, $\Omega\subset X\times \R$ open and bounded and $F\in C(\ov\Omega;X)$ a
    compact perturbation of the identity. If $F(x,t)\neq 0$ for all $x\in \partial \Omega_t$ with $t\in
    [0,1]$, then $t\mapsto d(F(\cdot,t),\Omega_t,0)$ is constant on $[0,1]$.
  \end{lem}
  
  For a proof of this well-known result we refer to Theorem 4.1 in~\cite{AmbMal:Nonlinear_analysis}.
  We mention that $\Omega_t:=\{ x\in X: (x,t)\in\Omega\}$, $t\in\R$, denote the slices of
  $\Omega$. Moreover, the projection of $\Omega$ onto the parameter space  will be denoted by 
  $\pr(\Omega):= \{t\in\R : \Omega_t \neq \emptyset\}$. Furthermore, we recall the following result, cf.
  Lemma~29.1~\cite{Deim}.
  
  \begin{lem}[Whyburn] 
    Let $(M,d)$ be a compact metric space, $A\subset M$ a component and $B\subset M$ closed such that $A\cap
    B=\emptyset$. Then there exist compact $M_A\supset A,M_B\supset B$ such that $M=M_A\cup M_B$ and
    $M_A\cap M_B=\emptyset$.
  \end{lem}

  \medskip

  \noindent {\it Proof of Theorem~\ref{thm:DanRab}:}\;  We assume that $\mathcal C$ is bounded. Exploiting the
  discreteness of the set of bifurcation points $\Xi\subset\mathcal T$, see
  Proposition~\ref{prop:unif_diff}, we find that $\mathcal C\cap \mathcal T$ is finite and
  \begin{align} 
    \ov{B_\rho(\mathcal C)}\cap  \Xi 
    &= \mathcal C\cap \mathcal T
    = \{ (x_{ij},\lambda_i): i=0,\ldots,k,\, j=0,\ldots,m_i\} \label{eq:C_cap_T}  
  \end{align}
  where $\rho>0$ is sufficiently small, $k\in\N_0$,$m_0,\ldots,m_k\in\N_0$ and the trivial solutions
  $(x_{ij},\lambda_i)$ are all different from each other. Without loss of generality we may assume
  $\lambda_0<\lambda_1<\ldots<\lambda_k$.
  Here, $B_\rho(\mathcal C):= \{(x,\lambda)\in X\times\R:
  \dist((x,\lambda),\mathcal C)<\rho\}$ denotes the open ball in $X\times\R$ around $\mathcal C$. 
  Replacing $F,\mathcal T$ by $\tilde F,\tilde{\mathcal T}$ given by  
  \begin{equation} \label{eq:change_coordinates}
    \tilde F(x,\lambda):=F(x,\lambda+\skp{\xi}{x}_{X'}),\qquad
    \tilde{\mathcal T}:=\{(x,\lambda-\skp{\xi}{x}_{X'}) : (x,\lambda)\in\mathcal T\}
  \end{equation}
  we may without loss of generality assume that $\mathcal C\cap \mathcal T$ does not have turning
  points so that $\xi=0$ is transverse to $\mathcal C\cap \mathcal T$. So it remains to
  prove~\eqref{eq:degree_balance_RabDan} in this special case.
  Due to the absence of turning points in $\mathcal C\cap \mathcal T$ we have (eventually after shrinking $\rho>0$)
  \begin{equation*}
    ( \mathcal T \cap \ov{B_\rho(\mathcal C)}\setminus \mathcal C )_{\lambda_i}
    = \emptyset \quad\text{ for } i=0,\ldots,k. \label{eq:partO_cap_T}
  \end{equation*}
  
  \medskip 
     
  Since $\mathcal S\cap \ov{B_\rho(\mathcal C)}$ is compact, we may invoke Whyburn's Lemma to get $\mathcal S\cap
  \ov{B_\rho(\mathcal C)} = K_1 \cup K_2$ where $K_1,K_2$ are disjoint compact sets such that
  $\mathcal C\subset K_1$ and $\mathcal S\cap \partial B_\rho(\mathcal C)\subset K_2$. 
  Then, for $0<\delta<\min\{\dist(K_1,K_2),\dist(K_1,\partial B_\rho(\mathcal C))\}$, the open set 
  $\mathcal O:= B_\delta(K_1)\subset B_\rho(\mathcal C)$ is a bounded open neighbourhood of $\mathcal C$ 
  satisfying
  \begin{itemize}
    \item[(i)] $\partial \mathcal O \cap \mathcal S =\emptyset$ and $(\partial \mathcal O\cap \mathcal
    T)_{\lambda_i} = \emptyset$ for $i=0,\ldots,k$,
    \item[(ii)] $\ov{\mathcal O} \cap \Xi = \mathcal C\cap \mathcal T$ is given by \eqref{eq:C_cap_T}.
  \end{itemize}
  Next we define for $\lambda\in\R\sm\{\lambda_0,\ldots,\lambda_k\}$  
  \begin{align*}
    \mu(\lambda) 
    &:= \lim_{r\to 0^+} d(F(\cdot,\lambda), B_r(\mathcal T)_\lambda,0),  
    \\
    \nu(\lambda)
    &:= \lim_{r\to 0^+}d(F(\cdot,\lambda),(\mathcal{O} \sm \ov{B_r(\mathcal T)})_\lambda,0).
  \end{align*}
  Then $\mu(\lambda)$ is well-defined since $\mathcal T_\lambda$ does not contain a bifurcation point 
  due to $\lambda\in\R\sm\{\lambda_0,\ldots,\lambda_k\}$  and (ii). Similarly, $\nu(\lambda)$ is well-defined
  in view of $\partial \mathcal O\cap \mathcal S=\emptyset$ by construction of $\mathcal O$.
  We now prove the following equalities for sufficiently small $\eps>0$:
    \begin{itemize}
      \item[(a)]  $\mu(\lambda_i-\eps)+\nu(\lambda_i-\eps) = 
      \mu(\lambda_i+\eps)+\nu(\lambda_i+\eps)$ for $i=0,\ldots,k$, 
      \item[(b)] $\nu(\lambda_i+\eps) = \nu(\lambda_{i+1}-\eps)$ for $i=0,\ldots,k-1$, 
      \item[(c)] $\nu(\lambda_0-\eps)=0$ and $\nu(\lambda_k+\eps)=0$.
    \end{itemize} 
   Choose $\eps>0$ such that $B_r(\mathcal T)_\lambda\subset \mathcal{O}_\lambda$ 
   for $|\lambda-\lambda_i|\leq \eps$. Then, by (i), $\partial O_\lambda$ does not contain any zeros of
   $F(\cdot,\lambda)$ whenever $|\lambda-\lambda_i|\leq \eps$. Hence, the additivity and the homotopy
   invariance of the degree yield
    \begin{align*}
      \mu(\lambda_i-\eps)+\nu(\lambda_i-\eps)
      &= d(F(\cdot,\lambda_i-\eps),\mathcal{O}_{\lambda_i-\eps},0) \\
      &= d(F(\cdot,\lambda_i+\eps),\mathcal{O}_{\lambda_i+\eps},0) \\
      &= \mu(\lambda_i+\eps)+\nu(\lambda_i+\eps).
    \end{align*}
    This proves (a). By property (ii) there is a sufficiently
    small $r>0$ such that solutions $(x,\lambda)\in \mathcal{O}$ with $\lambda_i+\eps\leq \lambda\leq \lambda_{i+1}-\eps$
    for $i=0,\ldots,k-1$ satisfy $(x,\lambda)\notin \ov{B_r(\mathcal T)}$ and $(x,\lambda)\notin \partial
    O\sm B_r(\mathcal T)$ follows from property (i).
    So, Lemma~\ref{lem:hom_inv} implies for $i=0,\ldots,k-1$
    \begin{align*}
      \nu(\lambda_i+\eps) 
      &=  d(F(\cdot,\lambda_i+\eps), (\mathcal{O} \sm \ov{B_r(\mathcal T)} )_{\lambda_i+\eps},0)   \\
      &=  d(F(\cdot,\lambda_{i+1}-\eps), (\mathcal{O} \sm \ov{B_r(\mathcal T)} )_{\lambda_{i+1}-\eps},0)   \\
      &= \nu(\lambda_{i+1}-\eps) 
    \end{align*}
    so that (b) is proved, too. Claim (c) follows again from Lemma~\ref{lem:hom_inv} and the
    fact that $\mathcal{O}$ is bounded. Indeed, for all
    $r>0$ we have 
    \begin{equation} \label{eq:Lem_munuI}
      d(F(\cdot,\lambda), (\mathcal{O} \sm \ov{B_r(\mathcal T)})_\lambda,0) 
      = d(F(\cdot,\lambda),\emptyset,0) 
      = 0
      \quad\; \text{if }\lambda\leq \lambda_*:= \inf \pr(\mathcal{O}). 
    \end{equation}
    As above, by (i) and (ii) we may choose $r>0$ such that all solutions $(x,\lambda)\in\ov{\mathcal
    O}$ with $\lambda_*\leq \lambda\leq \lambda_0-\eps$ satisfy
    $(x,\lambda)\notin \ov{B_r(\mathcal T)}$ as well as $(x,\lambda)\notin \partial O\sm B_r(\mathcal
    T)$.
    Hence, we get  from \eqref{eq:Lem_munuI}
    \begin{align*}
      \nu(\lambda_0-\eps)
      = d(F(\cdot,\lambda_0-\eps), (\mathcal{O}\sm  \ov{B_r(\mathcal T)} )_{\lambda_0-\eps},0)  
      = d(F(\cdot,\lambda_*),(\mathcal{O}\sm  \ov{B_r(\mathcal T)} )_{\lambda_*},0) 
      = 0. 
    \end{align*}
    The analogous reasoning gives $\nu(\lambda_k+\eps)=0$.
    
    \medskip
        
    From (a),(b),(c) we deduce
    \begin{align*}
      \sum_{i=0}^k  \big(\mu(\lambda_i+\eps)-\mu(\lambda_i-\eps)\big)  
   	  &= - \sum_{i=0}^k \big( \nu(\lambda_i+\eps)-\nu(\lambda_i-\eps) \big) \\
   	  &= \nu(\lambda_k-\eps) - 
   	  \sum_{i=0}^{k-1} \big( \nu(\lambda_i+\eps)-\nu(\lambda_i-\eps) \big)  \\ 
   	  &= \nu(\lambda_k-\eps) - 
   	  \sum_{i=0}^{k-1} \big( \nu(\lambda_{i+1}-\eps)-\nu(\lambda_i-\eps) \big)  \\ 
   	  &= \nu(\lambda_k-\eps) -  \nu(\lambda_k-\eps)+\nu(\lambda_0-\eps)  \\
   	  &= 0
    \end{align*}
    so that it remains to rewrite this identity in the form~\eqref{eq:degree_balance_RabDan}. 
    To this end we use that for  $|\lambda-\lambda_i|\leq \eps$ sufficiently
    small the slices $(\mathcal O\cap \mathcal T)_\lambda$ consist of precisely $m_i+1$ distinct points that converge to
    the points $x_{ij},j=0,\ldots,m_i$ as $\lambda\to \lambda_i$. Notice that at this point we use that
    none of these points is a turning point of $\mathcal T$.     
    Invoking the Leray-Schauder index formula
    (see for instance Theorem~8.10 in~\cite{Deim} or Lemma~3.19 in~\cite{AmbMal:Nonlinear_analysis}) we
    arrive for $\eps>0$ sufficiently small at
    \begin{align*}
      \mu(\lambda_i+\eps)-\mu(\lambda_i-\eps)
      &= \sum_{(x,\lambda_i+\eps)\in \mathcal O\cap \mathcal T} \ind(F(\cdot,\lambda_i+\eps),x) 
       -\sum_{(x,\lambda_i-\eps)\in \mathcal O\cap \mathcal T} \ind(F(\cdot,\lambda_i-\eps),x) \\
      &= \sum_{(x,\lambda_i+\eps)\in \mathcal O\cap \mathcal T} \ind(F_x(x,\lambda_i+\eps),0) 
       -\sum_{(x,\lambda_i-\eps)\in \mathcal O\cap \mathcal T} \ind(F_x(x,\lambda_i-\eps),0)  \\
      &= \sum_{j=0}^{m_i} \delta^*(x_{ij},\lambda_i;0). 
    \end{align*}
    These identities finally imply 
    \begin{align*}
      0 
      = \sum_{i=0}^k \big( \mu(\lambda_i+\eps) - \mu(\lambda_i-\eps)\big) 
      = \sum_{i=0}^k \sum_{j=0}^{m_i} \delta^*(x_{ij},\lambda_i;0) 
      = \sum_{(x,\lambda)\in\mathcal C\cap\mathcal T} \delta^*(x,\lambda;0).
    \end{align*} 
    \qed
    
%
  
%

  \section{Appendix B: On assumption (A2)}
  

  In this Section we motivate the assumption of locally uniform differentiability of 
  $F$ along the trivial solution $\mathcal T$ in the context of Proposition~\ref{prop:unif_diff}.   
%
%
  We provide an example for a not locally uniformly differentiable function $F:X\times\R\to X$ with
  $F(0,\lambda)=0$ for all $\lambda\in\R$ such that the set of bifurcation points is not discrete even though
  the set of degenerate solutions on $\mathcal T$ is. In particular, this shows that
  Proposition~\ref{prop:unif_diff} cannot hold without this assumption.
   
  \medskip

  The starting point for the construction of a counterexample is a differentiable function
  $f:\R\times\R\to\R$ with $f|_{\mathcal T}=0$ such that $f$ is not locally uniformly differentiable along
  $\mathcal T$, but $\lambda\mapsto f'(0,\lambda)$ is continuous. We define
  $M:= \max_{z\in\R}\sin^2(z)/z>0$ and its unique maximizer $z^*\approx 1.165561$. Then
  the following function has the above-mentioned properties: 
  $$
    f(x,\lambda) :=  x - M^{-1} \sin^2(x\lambda^{-3}) \lambda^3\quad\text{for }x\in\R,\lambda\neq 0,
    \qquad
    f(x,0) := x \quad\text{for }x\in\R
  $$
  Moreover, we have $f_x(0,\lambda)= 1$ for all $\lambda\in\R$ and $(0,0)$ is a bifurcation
  point because of $f(x_\lambda,\lambda)=0$ for all $\lambda\in\R$ where 
  $x_\lambda:= z^* \lambda^3$. Notice that there are no other nontrivial solutions. We conclude:
  \begin{equation*}
    (0,0) \text{ is a bifurcation point for }f(x,\lambda)=0 \text{ with }f_x(0,0)\neq 0.
  \end{equation*} 
  In other words $(0,0)$ is a nondegenerate bifurcation point for this equation in the sense we defined at
  the beginning of Section~\ref{sec:rabinowitz}. We stress  that
  this is possible due to the fact that $f$ is not locally uniformly differentiable along $\mathcal T$ and
  in particular not continuously differentiable in a neighbourhood of $\mathcal T$. In fact, one has 
  $x_\lambda\to 0$ as $\lambda\to 0$ and 
  $$
    \frac{f(x_\lambda,\lambda)-f(0,\lambda)-f_x(0,\lambda)x_\lambda}{x_\lambda}
    =  \frac{\sin^2(x_\lambda\lambda^{-3})}{M x_\lambda \lambda^{-3}}
    = \frac{\sin^2(z^*)}{Mz^*}
    = 1 
    \not\to 0 \quad\text{as }\lambda\to 0.
  $$ 
  In the next Lemma, this function is used for the construction of a counterexample.

  \begin{lem} \label{lem:counterexample}
    There is a differentiable function $F:\R\times\R\to \R$ satisfying $F(0,\lambda)=0$ for all
    $\lambda\in\R$ and such that the following holds for $\lambda_0\in\R$:
    \begin{itemize}
      \item[(i)] 
      The map $\lambda\mapsto F'(0,\lambda)$ is continuous on $\R$,
      \item[(ii)] $F_x(0,\lambda_0)=0$ and 
      $F_x(0,\lambda)(\lambda-\lambda_0)> 0$ for all $\lambda\neq \lambda_0$,
      \item[(iii)] there is a sequence of bifurcation points $(0,\lambda_n)$ for  $F(x,\lambda)=0$
      s.t. $\lambda_n\to \lambda_0$ as $n\to\infty$.
    \end{itemize}
  \end{lem}
  \begin{proof}
    W.l.o.g. we may assume $\lambda_0=0$. Let $f$ be defined as above, let $\chi\in C_0^\infty(\R)$ be a
    smooth cut-off function such that $\chi(z)=1$ for $|z|\leq \frac{1}{2}$, $0<\chi(z)<1$ for $|z|<1$ and
    $\chi(z)=0$ for $|z|\geq 1$ and define
    \begin{align*}
      F(x,\lambda)
      &:= \lambda \Big( \sum_{k\in\Z} \chi(a2^k(\lambda-2^{-k}))   f(x,\lambda-2^{-k}) + 
      \sum_{k\in\Z} \chi(a2^k(\lambda+2^{-k}))   f(x,\lambda+2^{-k}) \Big)
    \end{align*}
    where $a\in (2,3)$. Our aim is to verify the above-mentioned properties for
    $\lambda_n:=2^{-n}$. First let us mention that the $k$-th summand in the first series
    may not be zero only if $\lambda\in (\frac{a-1}{a}2^{-k},\frac{a+1}{a} 2^{-k})$. So, for all $\lambda>
    0$ we can find a small open neighbourhood of $\lambda$ and $k_\lambda\in\Z$ such that the second sum
    is zero and the $k$-th summand in the first series vanishes on this neighbourhood whenever
    $k\notin\{k_\lambda,k_\lambda +1\}$. Here, $a>1$ is used. The analogous reasoning applies to $\lambda<0$. 
    So the well-definedness and differentiability of $F$ at points $(0,\lambda)$ with $\lambda\neq 0$
    follows from the corresponding statements about $f$. Moreover, we have $F(0,\lambda)=0$ for all
    $\lambda\in\R$. Let us prove the claims (i)--(iii).

    \medskip

    \noindent {\it Proof of (i):} For $\lambda\neq 0$ we have $F_\lambda(0,\lambda)=0$ and
    \begin{equation}\label{eq:Fx_formula}
      F_x(0,\lambda) = \lambda \Big( \sum_{k\in\Z} \chi(a2^{k}(\lambda-2^{-k})) + \sum_{k\in\Z}
      \chi(a2^{k}(\lambda+2^{-k})) \Big).
    \end{equation}
    So (i) is proved once we show that $F'(0,0)$ exists with
    $F_x(0,0)=F_\lambda(0,0)=0$. Indeed, we have for $x,\lambda\to 0$
    \begin{align*}
      |F(x,\lambda)|
      &\leq |\lambda|  \sum_{k\in\Z} \big( |\chi(a2^{k}(\lambda-2^{-k}))| |f(x,\lambda-2^{-k})| 
        + |\chi(a2^{k}(\lambda+2^{-k}))| |f(x,\lambda+2^{-k})| \big)
       \\
      &\leq 2|\lambda| |x|  \sum_{k\in\Z} \big( |\chi(a2^{k}(\lambda-2^{-k}))|+
      |\chi(a2^{k}(\lambda+2^{-k}))|\big)  \\
      &\leq 4 |\lambda||x| \;=\; o(|x|+|\lambda|).
    \end{align*}
    Here we used that $\chi(a2^{k}(\lambda\pm 2^{-k}))$ is non-zero for at most two indices $k$.

    \medskip

    \noindent {\it Proof of (ii):}  For any given $\lambda>0$ we can choose $k\in\Z$ such that  
    $\lambda\in (\frac{a-1}{a}2^{-k},\frac{a+1}{a} 2^{-k})$. This is due to $a<3$. So
    $\chi(a2^{k}(\lambda-2^{-k}))>0$ and thus $F_x(0,\lambda)\lambda>0$ in
    view of ~\eqref{eq:Fx_formula}. The analogous reasoning applies to $\lambda<0$, which implies~(ii).

    \medskip

    \noindent {\it Proof of (iii):} We show that bifurcation occurs at $\lambda_n=2^{-n}$. Indeed,
    for $0<|\lambda|\leq \min\{\frac{a-1}{2a},\frac{a-2}{a}\}2^{-n}$, which is possible due to $a>2$, we
    have
    \begin{align*}
      a2^{k}|(2^{-n}+\lambda)-2^{-k}|
      &= a|2^{k-n}(1-2^n\lambda)-1|
      \geq a\Big(2\cdot \big(1-\frac{a-1}{2a}\big)-1\Big)
      = 1
    &&\text{if }k>n, \\
    a 2^k|(2^{-n}+\lambda)-2^{-k}|
    &\geq a\big(1-2^k(|\lambda|+2^{-n})\big)
    \geq a\Big(1-2^k\cdot \frac{2a-2}{a} 2^{-n}\Big)
    \geq a \big(1- \frac{a-1}{a}\big)
    = 1
    &&\text{if }k<n.
  \end{align*}
  The same inequalities hold for $a 2^k|(2^{-n}+\lambda)+2^{-k}|$ and all $k\in\Z$. So we get 
  \begin{align*}
    F(x_\lambda,2^{-n}+\lambda)
    &= (2^{-n}+\lambda)\Big[ \sum_{k\in\Z}  \underbrace{\chi(a2^{k}(2^{-n}+2^{-k}+\lambda))}_{=0}
    f(x_\lambda,2^{-n}+2^{-k}+\lambda) \\
    &\qquad\quad\quad + \sum_{k\in\Z, k\neq n}  \underbrace{\chi(a2^{k}(2^{-n}-2^{-k}+\lambda))}_{=0}
    f(x_\lambda,2^{-n}-2^{-k}+\lambda) \\
    &\qquad\quad\quad  + \chi(a2^n\lambda)  \underbrace{f(x_\lambda,\lambda)}_{=0}\Big] \\
    &= 0.
  \end{align*}
  Hence, $(0,\lambda_n)$ is a bifurcation point, which is all we had to show.
  \end{proof}   
  
%
%

  \section*{Acknowledgements} 
    The work on this project was supported by the Deutsche Forschungsgemeinschaft (DFG,
    German Research Foundation) through the Collaborative Research Center 1173.

\bibliographystyle{plain}
\bibliography{doc}

\end{document}